\documentclass{amsart}
\usepackage{times}
\usepackage{amsmath,amsthm, amssymb,bbm}
\usepackage{verbatim}

\newcommand{\R}{\mathbb{R}}

\newcommand{\one}{\mathbbm{1}} 

\newtheorem{thm}{Theorem}
\newtheorem{defn}{Definition}

\newtheorem{prop}{Proposition}
\newtheorem{lem}{Lemma}

\newtheorem{clm}{Claim}

\newcommand{\set}[1]{ \left\{ #1 \right\} }
\def\reals{\mathbb R}
\def\scriptl{\mathcal L}
\def\scripto{\mathcal O}
\def\scriptc{\mathcal C}
\def\naturals{\mathbb N}

\begin{document}

\address{
        Department of Mathematics\\
        University of California \\
        Berkeley, CA 94720-3840, USA}
\email{mchrist@berkeley.edu}
\email{taryn.flock@berkeley.edu}

\thanks{The authors were supported in part by NSF grant DMS-0901569.}

\title [Cases of Equality] {Cases of Equality \\ in Certain Multilinear Inequalities
\\ of Hardy-Riesz-Brascamp-Lieb-Luttinger Type}  
\author{ Michael Christ and Taryn C. Flock }

\date{August 20, 2013}

\maketitle

\begin{abstract}
Cases of equality in certain Hardy-Riesz-Brascamp-Lieb-Luttinger
rearrangement  inequalities  are characterized.
\end{abstract}

\begin{section}{Statement of result}

Let $m\ge 2$ and $n\ge m+1$ be positive integers.  For $j\in\set{1,2,\cdots,n}$
let $E_j\subset \R$  be Lebesgue measurable sets with positive, finite measures, 
and let $L_j$ be surjective linear maps $\R^m\to \R$. 
This paper is concerned with the nature of those $n$--tuples 
$(E_1,\cdots,E_n)$ of measurable sets that maximize expressions
\[I(E_1,\cdots,E_n)=\int_{\R^m}\prod_{j=1}^{n}\one_{E_j}(L_j(x))\,dx,\]
among all $n$--tuples with specified Lebesgue measures $|E_j|$.
Our results apply only in the lowest-dimensional nontrivial case, $m=2$, but apply for arbitrarily large $n$.

\begin{defn} 
A family $\set{L_j}$ of surjective linear mappings from $\reals^m$ to $\reals^1$
is nondegenerate  if for every set $S\subset\set{1,2,\cdots,n}$ of cardinality $m$,
the map $x\mapsto (L_j(x): j\in S)$ from $\reals^m$ to $\reals^S$ is a bijection.
\end{defn}


For any Lebesgue measurable set $E\subset\reals^1$ with finite Lebesgue measure,
$E^*$ denotes the nonempty closed\footnote{A more common convention is that $E^*$ should be open,
but this convention will be convenient in our proofs. If $E=\emptyset$ then $E^*=\set{0}$,
rather than the empty set, under our convention.}   interval centered at the origin satisfying $|E|=|E^*|$. 
Brascamp, Lieb, and Luttinger \cite{BLL} proved that among sets with specified measures,
the functional $I$ attains its maximum value when each $E_j$ equals $E_j^*$, that is,
\begin{equation} \label{mainIQ} I(E_1,\cdots,E_n)\leq  I(E_1^*,\cdots,E_n^*).  \end{equation} 
In this paper we study the uniqueness question and show that these are the only maximizing $n$--tuples, 
up to certain explicit symmetries of the functional,
in those situations in which a satisfactory characterization of maximizers can exist.


Inequalities of this type can be traced back at least to Hardy and to Riesz \cite{riesz}.
In the 1930s, Riesz and Sobolev independently showed that
\begin{equation*}
\iint_{\reals^k\times\reals^k} \one_{E_1}(x)\one_{E_2}(y)\one_{E_3}(x+y)\,dx\,dy
\le \iint_{\reals^k\times\reals^k} \one_{E_1^*}(x)\one_{E_2^*}(y)\one_{E_3^*}(x+y)\,dx\,dy
\end{equation*}
for arbitrary measurable sets $E_j$ with finite Lebesgue measures.
Brascamp, Lieb, and Luttinger \cite{BLL} later proved the more general result indicated above,
and in a yet more general form in which the target spaces $\reals^1$ are replaced by
$\reals^k$ for arbitrary $k\ge 1$, satisfying an appropriate equivariance hypothesis.

The first inverse theorem in this context, characterizing cases of equality, was established by Burchard
\cite{burchard}, \cite{burchardthesis}. The cases $n\le m$ are uninteresting, since 
$I(E_1,\cdots,E_n)=\infty$ for all $(E_1,\cdots,E_n)$ when $n<m$, and equality holds for all sets when $n=m$.
The results of Burchard \cite{burchardthesis} apply to the smallest nontrivial value of $n$ for given $m$, 
that is to $n=m+1$, but not to larger $n$. We are aware of no further progress in this direction since that time. 
This paper treats a situation at the opposite extreme of the spectrum
of possibilities, in which $m=2$ is the smallest dimension of interest, but the number
$n\ge 3$ of factors can be arbitrarily large.

Burchard's inverse theorem has more recently been applied to characterizations
of cases of equality in certain inequalities for the Radon transform and its generalizations
the $k$--plane transforms \cite{christradon},\cite{flockkplane}.
Cases of near but not exact equality for the Riesz-Sobolev inequality have been characterized still more recently
\cite{christRS1},\cite{christRS2}.

As was pointed out by Burchard \cite{burchard}, a satisfactory characterization of cases of
equality is possible only if no set $E_i$ is too large relative to the others.
This is already apparent for the trilinear expression associated to convolution,
\[I(E_1,E_2,E_3) =\iint \one_{E_1}(x)\one_{E_2}(y)\one_{E_3}(x+y)\,dx\,dy;\]
if $|E_3|>|E_1|+|E_2|$ and if $E_1,E_2$ are intervals, then equality holds
whenever $E_3$ is the union of an arbitrary measurable set with the algebraic sum of those two intervals.
 
Consider any expression $I(E_1,\cdots,E_n)$ where the integral is taken over $\reals^m$,
$E_j\subset\reals^1$, and $L_j:\reals^m\to\reals^1$ are linear and surjective.
Set $S_j=\{x\in\R^m:L_j(x)\in E_j\}$. Then $I(E_1,\cdots,E_n)$ is equal to the $m$--dimensional
Lebesgue measure of $\cap_j S_j$.
Define also 
\begin{equation} S_j^\star=\{x\in\R^m:L_j(x)\in E_j^*\}.  \end{equation}

\begin{defn}\label{defn:admissible} 
Let $(L_j: 1\le j\le n)$ be an $n$-tuple of surjective linear mappings from $\reals^m$ to $\reals$.
An $n$--tuple $(E_j: 1\le j\le n)$  of subsets of $\reals^1$
is admissible relative to $(L_j)$ if each $E_j$ is Lebesgue measurable and satisfies $0<|E_j|<\infty$,
and if there exists no index $k$ such that $S_k^\star$ contains an open neighborhood of $\bigcap_{j\neq k}S_j^\star$.  

$(E_j)$  is strictly admissible relative to $(L_j)$ if each set $E_j$ is Lebesgue measurable, $0<|E_j|<\infty$ for all $j$,
and there exists no index $k$ such that $S_k^\star$  contains $\bigcap_{j\neq k}S_j^\star$.
\end{defn}

Once the maps $L_j$ are specified,
admissibility of $(E_1,\dots,E_n)$ is a property only of the $n$--tuple of measures $(|E_1|,\dots,|E_n|)$. 
Its significance is easily explained.  Suppose that $(e_1,\cdots,e_n)$ is a sequence of
positive numbers such that an $n$-tuple of sets with these measures is not admissible.
The sets $E_j^*,S_j^\star$ are determined by $e_j$.
Choose an index $k$ such that $S_k^\star\supset\cap_{j\ne k}S_j^\star$.
For $j\ne k$ set $E_j=E_j^*$.
Choose the unique closed interval $I$ centered at $0$ such that the strip $S =\set{x: L_k(x)\in I}$
contains $\cap_{j\ne k} S_j^\star$, but $|I|$ is as small as possible among all such intervals. 
Choose $E_k$ to be the disjoint union of $I$ with an arbitrary set of measure
$|E_k|-|I|$. Then $I(E_1,\cdots,E_n) = I(E_1^*,\cdots,E_n^*)$, yet $E_k\setminus I$ is an artibrary set of
the specified measure.  Thus without admissibility, extremizing $n$-tuples are highly nonunique.


Admissibility and strict admissibility manifestly enjoy
the following invariance property. Let $\Phi$ be an affine automorphism of $\reals^m$,
and for $j\in\set{1,2,\cdots,n}$ let $\Psi_j$ be affine automorphisms of $\reals^1$.
Each composition $\Psi_j\circ L_j\circ\Phi$ is an affine mapping from $\reals^m$ to $\reals^1$.
Write $\Psi_j\circ L_j\circ\Phi(x) = \tilde L_j(x)+a_j$
where $\tilde L_j:\reals^m\to\reals^1$ is linear.
Define 
$\tilde E_j = \Psi_j(E_j)$ for all $j$.
Then $(E_j: 1\le j\le n)$ is admissible relative to $(L_j: 1\le j\le n)$
if and only if 
$(\tilde E_j: 1\le j\le n)$ is admissible relative to $(\tilde L_j: 1\le j\le n)$.
Strict admissibility is invariant in the same sense.

$A\bigtriangleup B$ will denote the symmetric difference of two sets.
$|E|$ will denote the Lebesgue measure of a subset of either $\reals^1$ or $\reals^2$.
We say that sets $A,B$ differ by a null set if $|A\bigtriangleup B|=0$.

The following theorem, our main result, characterizes cases of equality, in the situation in which $I(E_1,\cdots,E_n)$
is defined by integration over $\reals^2$ and $E_j\subset\reals^1$.

\begin{thm}\label{thm:main} Let $n\geq3$. 
Let $(L_i: 1\le i\le n)$ be a nondegenerate $n$-tuple of surjective linear maps  $L_i:\reals^2\to\reals^1$.
Let $(E_i: 1\le i\le n)$ be an admissible $n$--tuple of Lebesgue measurable subsets of $\reals^1$.
If $I(E_1,\cdots,E_n)=I(E_1^*,\cdots,E_n^*)$  then there exist a point $z\in\reals^2$,
and for each index $i$ an interval $J_i\subset\reals$,
such that $|E_i\bigtriangleup J_i|=0$ and the center point of $J_i$ equals $L_i(z)$. 
Conversely, $I(E_1,\cdots,E_n)=I(E_1^*,\cdots,E_n^*)$ in all such cases.
\end{thm} 

We conjecture that Theorem~\ref{thm:main} extends to arbitrary $m\ge 2$. 

The authors thank Ed Scerbo for very useful comments and copious suggestions regarding the exposition.
\end{section}

\begin{section}{On admissibility conditions }

For maps $L_j$ from $\reals^m$ to the simplest target space $\reals^1$, which is the subject of this paper, 
the most general case treated by Burchard \cite{burchardthesis} concerns
\begin{equation} \label{burchardform} 
\int_{\reals^m} \one_{E_0}(x_1+x_2+\cdots+x_m)\prod_{j=1}^m \one_{E_j}(x_j)\,
dx_1\cdots dx_m, \end{equation} 
where $m$ is any integer greater than or equal to $2$.
Cases of equality are characterized under the admissibility condition 
\begin{equation} \label{badmissible}
|E_i|\le \sum_{j\ne i}|E_j|\ \text{ for all $i\in\set{0,1,2,\cdots,m}$.} \end{equation}
Strict admissibility is the same condition, with inequality replaced by strict inequality for all $i$. 
This single case subsumes many cases, in light of the invariance property discussed above.


\begin{lem} \label{AEQ}
For the expression \eqref{burchardform},
admissibility in the sense \eqref{badmissible} is equivalent to admissibility 
in the sense of Definition~\ref{defn:admissible}.
Likewise, the two definitions of  strict admissibility are mutually equivalent.
\end{lem}

\begin{proof}
$S_0^\star=\{x: |\sum_{j=1}^n x_j|\le \tfrac12|E_0|\}$,
while for $j\ge 1$, $S_j^\star = \set{x: |x_j|\le \tfrac12 |E_j|}$.
Thus $|E_0|\ge\sum_{j=1}^n |E_j|$ if and only if 
\[S_0^\star \supset\set{x: |x_j|\le \tfrac12 |E_j|\ \text{ for all } 1\le j\le n}
=\cap_{j=1}^n S_j^\star.\]
Likewise, strict inequality is equivalent to inclusion of $\cap_{j=1}^n S_j^\star$ in the interior of $S_0^\star$.

For any $i\in\set{1,\cdots,n}$,
\[\cap_{j\ne i} S_j^\star = \set{x: |x_k|\le \tfrac12|E_k| \ \text{ for all }\ 
k\ne i\in\set{1,2,\cdots,n}} \bigcap \big\{x: |\sum_{l=1}^n x_j|\le\tfrac12 |E_0|\big\}\]
while
\[ S_i^\star =\set{x: |x_i|\le\tfrac12|E_i|}.\]
Therefore $|E_i|\ge \sum_{0\le j\ne i}|E_j|$ if and only if $S_i^\star\supset \cap_{0\le j\ne i} S_j^\star$,
and strict inequality is equivalent to inclusion of $\cap_{0\le j\ne i} S_j^\star$ in the interior of $S_i^\star$.
\end{proof}

The case $m=2$, $n=3$ of Theorem~\ref{thm:main} says nothing new.
Indeed, let $(L_j: 1\le j\le 3)$
be a nondegenerate family of linear transformations from $\reals^2$ to $\reals^1$. By making a
linear change of coordinates in $\reals^2$ we can make $L_1(x,y)\equiv x$ and $L_2(x,y)\equiv y$,
so that
\[I(E_1,E_2,E_3) = c\int_{\reals^2} \one_{E_1}(x)\one_{E_2}(y)\one_{E_3}(ax+by)\,dx\,dy\]
where $a,b$ are both nonzero. This equals
\[ c'\int_{\reals^2} \one_{E_1}(x/a)\one_{E_2}(y/b)\one_{E_3}(x+y)\,dx\,dy
= c'\int_{\reals^2} \one_{\tilde E_1}(x)\one_{\tilde E_2}(y)\one_{E_3}(x+y)\,dx\,dy \]
where $\tilde E_j$ are appropriate dilates and reflections of $E_j$.

We will need the following simple result concerning the stability of strict admissibility. 
\begin{lem}
Let $(L_j: 1\le j\le n)$ be a nondegenerate family of surjective linear mappings from $\reals^m$ to $\reals^1$. 
Let $(E_1,\cdots,E_n)$ be a strictly admissible $n$-tuple of Lebesgue measurable subsets of $\reals^1$.
There exists $\varepsilon>0$ such that any 
$n$-tuple $(E_1,\cdots,E_n)$ of Lebesgue measurable subsets of $\reals^1$ satisfying $\big|\,|E_j|-|F_j| \,\big|<\varepsilon$ for 
all $j\in\set{1,2,\cdots,n}$ is strictly admissible.
\end{lem}

\begin{proof} 
Suppose that no $\varepsilon$ satisfying the conclusion exists.
Then there exists a sequence of $n$-tuples $((E_{j,\nu}): \nu\in\naturals)$ such that $|E_{j,\nu}|\to |E_j|$
as $\nu\to\infty$, for each $j\in\set{1,2,\cdots,n}$, and such that for each $\nu\in\naturals$,
$(E_{n,\nu}: 1\le j\le n)$ is not admissible.

Let $E_{j,\nu}^*\subset\reals^1$ be the associated closed intervals centered at $0$.
Let \[S_{j,\nu}^\star=\set{x\in\reals^m: L_j(x)\in E_{j,\nu}^*}\] be the associated closed strips.
The failure of strict admissibility means that
for each $\nu$ there exists $J(\nu)$ such that $S_{J(\nu),\nu}^\star\supset \cap_{j\ne J(\nu)} S_{j,\nu}^\star$.
By passing to a subsequence we may assume that $J(\nu)\equiv J$ is independent of $\nu$.

Since $|E_{j,\nu}|\to|E_j|$, the closed strips $S_{j,\nu}^\star$  converge to the closed strips $S_j^\star$
as $\nu\to\infty$, in such a way that it follows immediately that $S_J^\star\supset \cap_{j\ne J}S_j^\star$.
Therefore $(E_1,\cdots,E_n)$ is not strictly admissible.
\end{proof}

\end{section}

\begin{section}{Truncation}

\begin{defn} 
Let $E\subset\reals^1$ have finite measure.
Let $\alpha,\beta>0$. If $\alpha+\beta \le |E|$ then
the truncation $E(\alpha,\beta)$ of $E$ is 
\begin{equation} E(\alpha,\beta)=E\cap[a,b]\end{equation}
where $a,b\in\reals$ are respectively the minimum and the maximum real numbers that satisfy
$$ |E\cap (-\infty,a]|=\alpha \text{  and  } |E\cap[b,\infty)|=\beta .$$
\end{defn}
In the degenerate case in which $\alpha+\beta=|E|$, $E(\alpha,\beta)$ has Lebesgue measure equal to zero, 
and may be empty or nonempty.
According to our conventions, $E(\alpha,\beta)^*=\set{0}$ in this circumstance, in either case.
This convention will be convenient below.

\begin{lem} \label {RL} 
Let $k\ge 1$.
Let $\set{E_i: i\in\{1,2,\cdots, k\}}$
be a finite collection of Lebesgue measurable subsets of $\reals^1$ 
with positive, finite Lebesgue measures.  
Let $\alpha,\beta > 0$, and suppose that $|E_i|\ge \alpha+\beta$ for each index $i$. 
If $\cap_{i=1}^k E_i(\alpha,\beta)\ne\emptyset$ then
\begin{equation} \label{rieszvariant} \int_\R\prod_{i=1}^k \one_{E_i}(y)\,dy\leq \alpha+\beta 
+\int_\R\prod_{i=1}^k\one_{{E}_i(\alpha,\beta)}(y)\,dy. \end{equation}
If $E_i$ are closed intervals and if $\cap_{i=1}^k E_i(\alpha,\beta)\ne\emptyset$
then equality holds in inequality \eqref{rieszvariant}.
\end{lem}

This generalizes a key element underpinning the work of Burchard \cite{burchard}, which in turn
is related, but not identical, to
the construction employed by Riesz \cite{riesz}.\footnote{Riesz considers only the case of three sets,
truncates all three in this fashion, uses only the case $\alpha=\beta$, and works directly with
the integral over $\reals^2$ which defines $I(E_1,\cdots,E_n)$, rather than with one-dimensional integrals.}

\begin{proof} 
For each index $i$, let $a_i,b_i\in\reals$ respectively be the smallest and the largest real numbers satisfying
$|E_i\cap(-\infty,a_i]|=\alpha$ and $|E_i\cap[b_i,\infty)|=\beta$.
Thus $E_i = [a_i,b_i]$.
Let $a=\max_i a_i$ and $b=\min_i b_i$.  Then $\bigcap_i{E}_i(\alpha,\beta) = (\cap_i E_i)\cap [a,b]$.
It is given that $\bigcap_i{E}_i(\alpha,\beta)$ is nonempty, so $a\le b$.

Thus
$$ \int_\R\prod_{i=1}^k\one_{E_i(\alpha,\beta)}(y)\,dy =  |\cap_{i}{E}_i(\alpha,\beta)|= |(\cap_{i} E_i)\cap [a,b]|.$$
Therefore 
\begin{align*}
\int_\R\prod_{i=1}^k \one_{E_i}(y)\,dy -\int_\R\prod_{i=1}^k\one_{{E}_i(\alpha,\beta)}(y)\,dy 
&= |(\cap_{i}E_{i})\setminus[a,b]| 
\\&= |(\cap_{i}E_{i})\cap (-\infty,a)|+ |(\cap_{i}E_{i})\cap (b,\infty)|. \end{align*}
Choose $l$ such that $a_l=a$. Then $(\cap_{i}E_{i})\cap (-\infty,a)\subset E_l\cap(-\infty,a)$ and hence 
$$ \left|(\cap_{i}E_{i})\cap (-\infty,a)\right|\leq |E_l\cap(-\infty,a)| = \alpha.$$ 
Similarly $ |(\cap_{i}E_{i})\cap (b,\infty)|\leq \beta$. 

For the converse, suppose that the $E_i$ are closed intervals, and that $\cap_i E_i(\alpha,\beta)\ne\emptyset$. 
Then $\cap_i E_i(\alpha,\beta)= [a,b]$ where $a\le b$, as above.
In the same way, $\cap_i E_i = [a^\star,b^\star]$ where $a^\star$ is the 
maximum of the left endpoints of the intervals $E_i$, and $b^\star$
is the minimum of their right endpoints. Obviously $a^\star = a-\alpha$ and $b^\star=b+\beta$.
\end{proof}

The next lemma is evident.
\begin{lem} \label{lemma:evident}
Let $0\le \alpha,\beta<\infty$. 
Let $\set{I_k}$ be a collection of closed bounded
subintervals of $\reals$ satisfying $|I_k|\ge \alpha+\beta$.
Suppose that $\cap_k I_k(\alpha,\beta)\ne\emptyset$, and that $J$ is a closed subinterval of $\reals$ satisfying
$J(\alpha,\beta)\supset \cap_k I_k(\alpha,\beta)$. Then $J\supset \cap_k I_k$. 
\end{lem}

\end{section}

\begin{section}{Deformation} 

We change notation: The number of sets $E_j$ will be $n+1$, and the index $j$ will run through $\set{0,1,\cdots,n}$.
The index $j=0$ will have a privileged role.

Consider a functional \[I(E_0,\cdots,E_n) = \int_{\reals^2}\prod_{j=0}^n \one_{E_j}(L_j(x))\,dx,\]
with $\set{L_j: 0\le j\le n}$ nondegenerate.  The invariance under changes of variables noted above,
together with this nondegeneracy, make it possible to bring
this functional into the form 
$$ I(E_0,\cdots,E_n)=c \int_{\R}\one_{E_0}(x)\int_{\R}\prod_{j=1}^{n}\one_{E_j}(y+t_jx)\,dy\,dx $$ 
where $c$ is a positive constant, and the $t_j$ are pairwise distinct. 
This is accomplished
by means of a linear change of variables in $\reals^2$ together with linear changes of variables in 
each of the spaces $\reals^1_j$ in which the sets $E_j$ lie. 
The sets $E_j$ which appear here are images of the original sets $E_j$ under invertible linear mappings of $\reals^1_j$,
but equality holds in the inequality \eqref{mainIQ} for this rewritten expression $I(E_0,\cdots,E_n)$
if and only if it holds for the original expression, and the property of admissibility is preserved.

With $I(E_0,\cdots,E_n)$ written in this form,
\begin{align*} S_0^\star&=\{(x,y)\in\R^2: |x|\le \tfrac12|E_0|\}  
\\ S_j^\star&=\{(x,y)\in\R^2: |y+t_jx|\le \tfrac12 |E_j| \} \text{  for $1\leq j\leq n$.}  \end{align*}
Let $\pi:\reals^2\to\reals^1$ be the projection $\pi(x,y)=x$.  Define
\begin{align*} &E_j(r) = E_j(\tfrac12 r,\tfrac12 r)\ \text{ for  $j\ge 1$ and $0<r \le |E_j|$}
\\&E_j(0)=E_j
\\ &E_0(r)\equiv E_0. \end{align*}
Thus $|E_j(r)| = |E_j|-r$ for $j\ge 1$.
Let $S_j^\star(r)$ be the associated strips; $S_0^\star(r)= S_0^\star$
while for $j\ge 1$, \[S_j^\star(r)=\set{(x,y)\in\reals^2: |y+t_j x|\le \tfrac12|E_j|-\tfrac12 r}\]
for $0\le r\le \min_j|E_j|$.
Thus if $j\ge 1$ and $r=|E_j|$ then  $S_j^\star(r)$ is a line in $\reals^2$.

The cases $n\ge 3$ of the next lemma will later be used to prove Theorem~\ref{thm:main} by induction on $n$.
\begin{lem} \label{Tok}
Let $n\ge 2$. 
Let $\set{E_j: 0\le j\le n}$ be a strictly admissible family of $n+1$
Lebesgue measurable subsets of $\reals^1$.
Then there exists $\bar r\in(0,\min_{1\le j\le n} |E_j|)$ such that 
\begin{align*}
&(E_j(\bar r): 0\le j\le n) \ \text{ is admissible} 
\\ &S_0^\star \supset \cap_{j\ge 1} S_j^\star(\bar r).
\end{align*}
\end{lem}

The second conclusion says in particular that $(E_j(\bar r): 0\le j\le n)$ fails to be strictly admissible.
Because admissibility is a property of the measures of sets only with no reference to their geometry,
Lemma~\ref{Tok} concerns deformations of intervals centered at $0$ and of associated strips, not of more general sets.

\begin{proof}
Define $\bar r$ to be the infimum of the set of all $r\in[0,\min_{k\ge 1} |E_k|]$ for which
$(E_j(r): 0\le j\le n)$ fails to be strictly admissible. 
If $r=\min_{k\ge 1{k\ge 1}} |E_k|=|E_i|$ then $|E_i(r)|=0$ and therefore
$(E_j(r): 0\le j\le n)$ is not strictly admissible. Thus $\bar r$ is defined as the infinmum of a nonempty set,
and $0\le \bar r\le \min_{k\ge 1}|E_k|$.

Since $(E_0,\cdots,E_n)=(E_0(0),\cdots,E_n(0))$ is strictly admissible,
and since strict admissibility is stable under small perturbations of the type under consideration, 
the $n+1$-tuple $(E_0(r),\cdots,E_n(r))$ is strictly admissible for all sufficiently small $r\ge 0$.
Therefore $\bar r>0$.


Consequently the definition of $\bar r$ implies one of two types of degeneracy: 
Either $|E_l^*(\bar r)|=0$ for some $l\ge 1$, or there exists $i\in\set{0,1,\cdots,n}$ such that
\begin{equation} \label{eq:badi}
\text{$S_i^\star(\bar r)\supset \cap_{j\ne i} S_j^\star(\bar r)$.} \end{equation}

\begin{clm}
The inclusion \eqref{eq:badi} must hold for at least one index $i\in\set{0,1,\cdots,n}$.  
\end{clm}

\begin{proof}
If not, then the other alternative must hold; 
there exists an index $l$ such that $|E_l^*(\bar r)|=0$. In that case, $S_l^\star(\bar r)$
is by definition equal to the line $\set{(x,y): y+t_lx=0}$, which contains $0$. 
For each index $j\ne l$, the intersection of $S_j^\star(\bar r)$ with $\scriptl$ is a nonempty closed interval
of finite nonnegative length, centered at $0$.
Choose $i\ne l$ for which the length of $S_i^\star(\bar r)\cap\scriptl$
is maximal. Then $S_i^\star(\bar r)$ contains $S_i^\star(\bar r)\cap\scriptl$, which in turn
contains $S_j^\star(\bar r)\cap\scriptl$ for every $j\notin\set{i,l}$. Therefore \eqref{eq:badi} holds for this index $i$.
\end{proof}

Let \[K = \cap_{j=1}^n S_j^\star(\bar r),\] 
which is a nonempty balanced convex subset of $\reals^2$. 
$K$ is compact, by the nondegeneracy hypothesis, since $E_j^*$ are compact intervals.

$\pi(K)\subset\reals$ is a compact interval centered at $0$, as is $E_0^*$. Therefore\footnote{This 
apparently innocuous step is responsible for the restriction $m=2$ in our main theorem.}
$\pi(K)\subset E_0^*$, or $E_0^*\subset\pi(K)$.

\begin{clm} 
If $\pi(K)\supset E_0^*$ and if an index $i$ satisfies \eqref{eq:badi}, then $i=0$.  \end{clm}

\begin{proof}
Suppose that $\pi(K)\supset E_0^*$ and that $i\ne 0$ satisfies \eqref{eq:badi}.
For $1\le j\le n$ define the closed intervals 
\begin{equation} J(x,j,r) = \set{y\in\reals^1: (x,y)\in S_j^\star(r)}\subset\reals^1.\end{equation} 
For any $x\in\pi(K)$, these intervals have at least one point in common. 
Since $S_i^\star(\bar r)\supset \cap_{j\ne i} S_j^\star(\bar r)$, 
\[J(x,i,\bar r)\supset \cap_{j\ne i} J(x,j,\bar r)\ \text{ for any } x\in E_0^*.\]
Therefore by Lemma~\ref{lemma:evident}, 
\begin{equation} J(x,i,0)\supset \cap_{1\le j\ne i} J(x,j,0)\ \text{ for all $x\in E_0^*$.}\end{equation}


Since $S_0^\star=\pi^{-1}(E_0^*)$ it then follows that
\[
S_i^\star \supset 
S_i^\star\cap \pi^{-1}(E_0^*) \supset  \cap_{1\le j\ne i} S_j^\star\cap \pi^{-1}(E_0^*) 
= \cap_{0\le j\ne i} S_j^\star,\]
contradicting the hypothesis that $(E_0,\cdots,E_n)$ is strictly admissible.
\end{proof}


\begin{clm}
$\pi(K)$ cannot properly contain $E_0^*$. \end{clm}

\begin{proof}
Suppose that $\pi(K)$ properly contains $E_0^*$. By the preceding Claim, \eqref{eq:badi} holds for $i=0$.
Let $x\in \pi(K)\setminus E_0^*$.  There exists $y\in\reals$ such that $(x,y)\in K$.
Since $x\notin E_0^*$, $(x,y)\notin S_0^\star =\pi^{-1}(E_0^*)$. 
Therefore $K=\cap_{j\ge 1} S_j^\star(\bar r)$ is not contained in $S_0^\star = S_0^\star(\bar r)$,
contradicting \eqref{eq:badi}.
\end{proof}

\begin{clm} $\pi(K)$ is not properly contained in $E_0^*$.  \end{clm}

\begin{proof}
If $\pi(K)$ is properly contained in $E_0^*$, then it is contained in the interior of $E_0^*$, 
since each of these sets is a closed interval centered at $0$.
Consequently $K$
is contained in the interior of 
$\pi^{-1}(E_0^*)=S_0^\star=S_0^\star(\bar r)$;
that is, $\cap_{j\ge 1} S_j^\star(\bar r)$ is contained in the interior of $S_0^\star$.
Therefore for every $r'<\bar r$ sufficiently close to $\bar r$, 
$\cap_{j\ge 1} S_j^\star(r')$ is contained in $S_0^\star$.
Thus $(E_0(r'),\cdots,E_n(r'))$ fails to be strictly admissible.
This contradicts the definition of $\bar r$ as the infimum of the set of all $r$
for which $(E_0(r),\cdots,E_n(r))$ fails to be strictly admissible.
\end{proof}

Combining the above four claims, we conclude that \eqref{eq:badi} holds for $i=0$ and for no other
index, and that $\pi(K)=E_0^*$.

\begin{clm} $|E_j(\bar r)|>0$ for every index $j\in\set{0,1,\cdots,n}$.  \end{clm}

\begin{proof}
If $|E_l(\bar r)|=0$ then since $E_0(\bar r)=E_0$, the index $l$ cannot equal $0$.
$S_l^\star(\bar r)$ is the line $\scriptl=\set{(x,y): y+t_lx=0}$.
For each $j\ne l$, $S_j^\star(\bar r)\cap\scriptl$ is a closed subinterval of $\scriptl$ centered at $0$.
Therefore $K$ is equal to the smallest of these subintervals.

Since $\pi(K)=E_0^*$, and since $\pi:\scriptl\to\reals$ is injective,
$K$ must equal $\scriptl\cap S_0^\star = S_l^\star(\bar r)\cap S_0^\star$.
Therefore $S_j^\star(\bar r)\cap\scriptl\supset S_0^\star(\bar r)\cap\scriptl$.
Therefore every $i\notin\set{0,l}$ satisfies \eqref{eq:badi}.
Since $n\ge 2$ there are at least three indices $0\le i\le n$, so there exists at least one index $i\notin\set{0,l}$.
But we have shown that the only such index is $i=0$, so this is a contradiction.
\end{proof}


\medskip
To conclude the proof of Lemma~\ref{Tok}, it remains to show that
$(E_0(\bar r),\cdots,E_n(\bar r))$ must be admissible.
We have shown that $|E_j(\bar r)|>0$ for all $j$.
The failure of admissibility is a stable property for sets with  positive measures, 
so if $(E_0(\bar r),\cdots,E_n(\bar r))$ were not admissible then there would exist $0<r<\bar r$ for which
$(E_0(r),\cdots,E_n(r))$ was not admissible, contradicting the minimality of $\bar r$.
\end{proof}



\end{section}

\begin{section}{Conclusion of the Proof} 

The proof of Theorem~\ref{thm:main} proceeds by induction on 
the degree of multilinearity of the form $I$, that is, on the number of sets appearing in
$I(E_1,\cdots,E_n)$.
The base case $n=3$ is a restatement of the one-dimensional case of Burchard's theorem, 
in its invariant form, since the two definitions of admissibility are equivalent.

Assuming that the result holds for expressions involving $n$ sets $E_j$, we will prove it
for expressions involving $n+1$ sets. 
Let $(E_0,\cdots,E_n)$ be any admissible $n+1$--tuple of sets satisfying $I(E_0,\cdots,E_n)=I(E_0^*,\cdots,E_n^*)$.

Consider first the case in which $(E_j: 0\le j\le n)$ is not strictly admissible.
Then there exists $i$ such that $S_i^\star\supset \cap_{j\ne i} S_j^\star$. By permuting the indices,
we may assume without loss of generality that $i=0$. 
Then
\begin{align*} I(E_0,\cdots,E_n) \le I(\reals,E_1,\cdots,E_n) \le I(\reals,E_1^*,\cdots,E_n^*) 
= I(E_0^*,\cdots,E_n^*),
\end{align*}
so $I(\reals,E_1,\cdots,E_n) = I(\reals,E_1^*,\cdots,E_n^*)$.

Defining \[J(E_1,\cdots,E_n) = I(\reals,E_1,\cdots,E_n),\] we have  
$J(E_1,\cdots,E_n) = J(E_1^*,\cdots,E_n^*)$.
Now $(E_1,\cdots,E_n)$ is admissible relative to $\set{L_j: 1\le j\le n}$.
For if not, then there would exist $k\in\set{1,2,\cdots,n}$
for which $S_k^\star$ properly contained $\cap_{1\le j\ne k} S_j^\star$. 
Since $S_0^\star \supset \cap_{j\ge 1} S_j^\star$,
\[\cap_{1\le j\ne k} S_j^\star = S_0^\star\cap(\cap_{1\le j\ne k} S_j^\star). \]
so $S_k^\star$ would properly contain $\cap_{0\le j\ne k} S_j^\star$,
contradicting the hypothesis that $(E_0,\cdots,E_n)$ is admissible.

By the induction hypothesis, equality in the rearrangement inequality for $J$ can occur only if
$E_j$ differs from an interval by a null set, for each $j\ge 1$.
Moreover, there must exist a point $z\in\reals^2$ such that for every $j\in\set{1,2,\cdots,n}$, 
$L_j(z)$ equals the center of the interval corresponding to $E_j$.

For $j\ge 1$, replace $E_j$ by the unique closed interval which differs from $E_j$ by a null set.
By an affine change of variables in $\reals^2$, we can write $I(E_0,\cdots,E_n)$ in the form
\begin{equation} \label{Iconvenientform} c \int \one_{E_0}(x)\int \prod_{j=1}^n \one_{E_j}(y+t_jx)\,dy\,dx\end{equation}
where $c\in(0,\infty)$ and $t_j\in\reals$, and now for each $j\ge 1$,
$E_j$ is an interval centered at $0$. 
The inner integral defines a nonnegative function $F$ of $x\in\reals$
which is continuous, nonincreasing on $[0,\infty)$, even,
and has support equal to a certain closed bounded interval centered at $0$.
The condition that $(E_0,\cdots,E_n)$ is admissible but $S_0^\star\supset\cap_{j=1}^n S_j^\star$
means that this support is equal to the closed interval $E_0^*$. 
Among sets $E$ satisfying $|E|=|E_0|$,
$\int_E F<\int_\reals F$ unless $E$ differs from $E_0^*$ by a null set. 
We have thus shown that in any case of nonstrict admissibility, 
all the sets $E_j$ differ from intervals by null sets, and the centers $c_j$ of these intervals 
are coherently situated, in the sense that $c_j=L_j(z)$ for a common point $z\in\reals^2$.

Next consider the case in which $(E_0,\cdots,E_n)$ is strictly admissible.
Change variables to put $I(E_0,\cdots,E_n)$ into the form \eqref{Iconvenientform}.
This replaces the sets $E_j$ by their images under certain invertible linear transformations,
but does not affect the validity of the two conclusions of the theorem.

Let $\bar r$ be as specified in Lemma~\ref{Tok}.
Set $\tilde E_j = E_j(\bar r)$, and recall that $\tilde E_0=E_0$.
Let $\tilde S_j^\star$ be the strips in $\reals^2$ associated to the rearrangements $\tilde E_j^*$.
By Lemma~\ref{RL}, 
\[\int_{\R}\prod_{j=1}^n\one_{E_j}(y+t_jx)dy \leq \bar r + \int_{\R}\prod_{j=1}^n\one_{\tilde{E_j}}(y+t_jx)\,dy\]
for each $x\in E_0$.
Multiplying both sides by $\one_{E_0}(x)$ and integrating with respect to $x$ gives 
\[\int_{\R}\one_{E_0}(x)\int_{\R}\prod_{i=1}^n\one_{E_j}(y+t_jx)\,dy\,dx  
\leq \bar r|E_0|+\int_{\R}\one_{E_0}(x) \int_{\R}\prod_{i=1}^n\one_{\tilde{E_j}}(y+t_jx)\,dy\,dx.\]

Thus 
\begin{equation} \label{Et1} 
	I(E_0,\ldots,E_n)\leq \bar r|E_0|+I(E_0,\tilde{E}_1,\ldots,\tilde{E}_n). 
\end{equation} 
By the general rearrangement inequality applied to the $n+1$--tuple $(E_0,E_1,\dots,E_n)$,
\begin{equation} \label{Et2} 
 \bar r|E_0|+I(E_0,\tilde{E}_1,\ldots,\tilde{E}_n) \leq \bar r|E_0|+I(E_0^*,\tilde{E}_1^*,\ldots,\tilde{E}_n^*).  
\end{equation} 

Since $(\tilde E_j: 0\le j\le n)$ is admissible,
for each $x\in E_0$ there exists $y$ such that $(x,y)\in \cap_{j\ge 1} \tilde S_j^\star$.
Therefore by the second conclusion of Lemma~\ref{RL},
\[\int_{\R}\prod_{i=1}^n\one_{E_j^*}(y+t_jx)\,dy = \bar r+ \int_{\R}\prod_{i=1}^n\one_{\tilde E_j^*}(y+t_jx)\,dy.\]
Integrating both sides of this inequality with respect to $x\in E_0^*$ gives
\begin{equation} \label{Et3} 
	I(E_0^*,E_1^*,\ldots,E_n^*)=\bar r|E_0^*|+I(E_0^*,\tilde{E}_1^*,\ldots,\tilde{E}_n^*) .
\end{equation} 

Combining \eqref{Et1}, \eqref{Et2}, and \eqref{Et3} yields 
\begin{multline*}
I(E_0,\ldots,E_n)\leq \bar r|E_0|+I(E_0,\tilde{E}_1,\ldots,\tilde{E}_n) 
\leq \\ \bar r|E_0|+I(E_0^*,\tilde{E}_1^*,\ldots,\tilde{E}_n^*) 
=I(E_0^*,E_1^*,\ldots,E_n^*)
\end{multline*}  
We are assuming that
$I(E_0,E_1,\ldots,E_n)=I(E_0^*,\tilde{E}_1^*,\ldots,\tilde{E}_n^*)$, 
so equality holds in each inequality in this chain. Hence
\[ I(E_0,\tilde{E}_1,\ldots,\tilde{E}_n) =I(E_0^*,\tilde{E}_1^*,\ldots,\tilde{E}_n^*). \] 

Thus the $n+1$--tuple $(E_0,\tilde E_1,\cdots,\tilde E_n)$ is admissible but not strictly admissible,
and achieves equality in the inequality \eqref{mainIQ}. This situation was analyzed above.
Therefore we conclude that $E_0$ coincides with an interval, up to a null set.

The same reasoning can be applied to $E_j$ for all $j$, by permuting the indices,
so each of the sets $E_j$ is an interval up to a null set. 
In this case (returning to the above discussion in which the index $j=0$
is singled out), each interval $E_j$ has the same center as $E_j(\bar r)$.
The discussion above has established that the centers of the intervals $E_j(\bar r)$ are coherently situated. 
\end{section}

\end{document}